\documentclass[12pt]{amsart}

\usepackage{graphics, amsmath, amsfonts, amssymb, amscd, mathrsfs}

\newtheorem{theorem}{Theorem}[section]
\newtheorem{proposition}[theorem]{Proposition}

\newtheorem{corollary}[theorem]{Corollary}

\theoremstyle{definition}

\newtheorem{counterexample}[theorem]{Counterexample}

\newtheorem{Remark}[theorem]{Remark}

\newcommand{\B}{\mathbb{B}}

\newcommand{\C}{\mathbb{C}}
\newcommand{\D}{\mathbb{D}}

\newcommand{\R}{\mathbb{R}} 

\newcommand{\T}{\mathbb{T}}
\newcommand{\PSH}{\operatorname{PSH}}
\newcommand{\USC}{\operatorname{USC}}
\newcommand{\Id}{\operatorname{Id}}

\title{ Boundary relative extremal functions}

\author{Ibrahim K. Djire }
\address{Jagiellonian University, Department of Mathematics}
\email{Ibrahim.Djire@im.uj.edu.pl}

\author{ Jan Wiegerinck  }
\address{Korteweg--de Vries Institute, Universiteit van Amsterdam, Science Park  105-107, Amsterdam} 
\email{J.J.O.O.Wiegerinck@uva.nl}

 
\begin{document}
\begin{abstract}

We study the relation between certain alternative definitions of the boundary relative extremal function.
   For various domains we give an affirmative answer to the question of Sadullaev, \cite{Sa}, whether these extremal functions are equal. 
\end{abstract}
\maketitle

\section{Introduction}
Let $D\subset \C^n$ be a smoothly bounded domain, $A\subset \partial D$, and let $\PSH(D)^-$ stand for the family of non-positive plurisubharmonic functions on $D$. For $u\in\PSH(D)^-$ as usual
\[u^*(z)=\limsup_{\zeta\to z,\zeta\in D} u(\zeta)\qquad (z\in \overline D).\]
Sadullaev studied the first three of the following boundary extremal functions.  For $z\in D$ consider

\begin{enumerate}
\item  $\omega_1(z,A,D)=\omega^c(z,A,D)=\sup\{ u(z), u\in \PSH(D)^-\cap C(\overline D), u|A\leq-1\},$
\item  $\omega_2(z,A,D)=\omega(z,A,D)=\sup\{ u(z), u\in \PSH(D)^-, u^*|A\leq-1\},$
\item  $\omega_3(z,A,D)=\omega^n(z,A,D)=\sup\{u(z), u\in \PSH(D)^-, \limsup_{z\rightarrow\zeta, z\in n_{\zeta}} u(z)\leq-1 \mbox{ for  } \zeta\in A\}$ where $n_{\zeta}$ is the inward  normal to $\partial D$ at $\zeta$,
\item  $\omega^R(z,A,D)=\sup\{u(z), u\in \PSH(D)^-, \limsup_{r\rightarrow 1^-} u(rz)\leq-1, z\in A\},$ if $D$ is strongly star shaped with respect to the origin.
\end{enumerate} 
Actually, smoothness is needed only to define $\omega^n$. 
It is clear that 
\[\omega_1(.,A,D)\leq\omega_2(.,A,D)\leq\omega_3(.,A,D).\]
This paper is motivated by the following question (Problem 27.4 in \cite{Sa}): Suppose $A\subset \partial D$ is closed, for what $i, j$ is $\omega_i^*(z,A, D)\equiv \omega_j^*(z,A, D)$?

The answer apparently depends on the geometry and  convexity properties of $D$ and the choice of the compact set $A\subset\partial D$. For instance we showed in \cite{IbJa} that Sadullaev's question has a positive answer when  $D$ is a smooth pseudoconvex Reinhardt domain and $A$ is multi-circular. The result in \cite{IbJa} exploits the relation between  relative extremal functions and convex functions in a Reinhardt domain.

In the present paper we answer in Section 3 the question affirmatively for ellipsoidal domains $D_H$, which are biholomorphic to the unit ball via a linear transformation.  Here we exploit an idea of Wikstr\"om \cite{Wi} and use Edwards' duality theorem. In Section 4 we show equality for circular sets $A$ in the boundary of circular, strongly star shaped domains $D$. We attempted to use the version of Edwards' theorem in \cite{GoPePo} and found that their result is not correct. In the appendix we give two pertaining counterexamples.

 We denote the open unit disc in $\C$ by $\D$, its boundary by $\T$ and the unit ball in $\C^n$ ($n\geq 2$) by $\B$.
 Some basic properties of the boundary relative extremal function are given in \cite{IbJa, EdSig, Edi, Po, Sa} (\cite{EdSig} appeared as \cite{EdSi} but the preprint is more relevant). 
  Depending on the way the boundary is approached, plurisubharmonic function may have different boundary values.   Wikstr\"om considered the compact set $A=\T\times\{0\}$ and the function $u\in \PSH(\B)$:
\[u(z)= \log\frac{|z_2|^2}{1-|z_1|^2}.\]
He showed that $u^*=0$ on $A$. The radial limit of $u$, $u^R=-\infty$ on $A$ and the non-tangential limit of $u$, $u^{\alpha}=\log(1-1/2\alpha)$ on $A$ \cite[Example 5.5]{Wi}. We recall the definition of $u^{\alpha}.$ If $\alpha>1$ and $z_0\in\partial \B$ we put $$D_{\alpha}(z_0)=\{z\in\B : |1-\big\langle z, z_0\big\rangle|<\alpha(1-|z|^2)\},$$ $$u^{\alpha}(z_0)=\limsup_{z\rightarrow z_0, z\in D_{\alpha}(z_0)} u(z).$$

  \section{Notations and definitions}  Let $D=\{\rho<0\}$ be a domain in  $\C^n$ with  $C^1$-boundary and defining function $\rho$. For $z\in \overline D$ and $t\in \R$ let
 \[n(z,t)=z-t\left(\frac{\partial\rho}{\partial \overline z_1}(z),\ldots, \frac{\partial\rho}{\partial \overline z_n}(z)\right).\] 
 If $z\in\partial D$  the normal line $n_z$ passing through $z$ is parametrized by $\{n(z,t), t\in \R\}$.

 Let $u :D\rightarrow \R\cup\{-\infty\}$ be bounded from above  and $z\in \partial D$ we define  $u^n$ at $z$ as     $$u^n(z)=\limsup_{t\downarrow 0} u\circ n(z,t).$$
Extend $u^n$ to $\overline D$ by setting $u^n(z)=u(z)$ if $z\in D$. Recall that $D$ is called \emph{strongly star shaped with respect to the origin} if $r\overline D\subset D$ for $r\in ]0,1[$. If $D$ is strongly star shaped with respect to the origin, then for $z\in\partial D$  set  $u^R(z)=\limsup_{r\uparrow 1} u(zr).$ Extend $u^R$ to $\overline D$ by setting $u^R(z)=u(z)$ if $z\in D$. Throughout the paper by strongly star shaped we mean strongly star shaped with respect to the origin. Let $M(D)$ be the set of Borel probability measures  with compact support on $\overline D.$ For $z\in \overline D $ we consider four classes of positive measures
\begin{enumerate}
 \item $J_z=J_z(\overline D)=\{ \mu\in M(D) : u(z)\leq \int_{\overline D} u\, d\mu \text{ for all } u\in\PSH(D)\cap\USC(\overline D) \} $
   \item $ J^c_z=J^c_z(\overline D)=\{ \mu\in M( D) : u(z)\leq \int_{\overline D} u \, d\mu \text{ for all } u\in\PSH(D)\cap C(\overline D)\} $
\item $J^n_z=  \{ \mu \in M(D) : u^n(z)\leq \int_{ \overline D} u^n\,  d\mu \text{ for all } u\in\PSH(D), \sup_{\overline D} u^n<\infty\}    $
\item $J^R_z=  \{ \mu \in M(D) : u^R(z)\leq \int_{ \overline D} u^R\,  d\mu \text{ for all } u\in\PSH(D), \sup_{\overline D} u^R<\infty\},$ in case $D$ is strongly star shaped with respect to the origin.
\end{enumerate}
Clearly for $z\in D,$  $J^n_z, J^R_z\subset J_z\subset J_z^c$. Wikstr\"om studied these classes and proved that $J=J^c=J^R$ on $D$ if  $D$ is strongly star shaped, see \cite[Proposition 5.4]{Wi}. If $U\subset\overline D,$ $\chi_U$ denotes the characteristic function of $U.$ 


\section{Applications of Wikstr\"om's results}
We use equalities between different classes of Jensen measures to prove the equivalence of different definitions. This is done by applying Edwards' theorem to the convex cone $\PSH(D)\cap\USC(\overline D)$ and the associated Jensen measures $J_z$.

\begin{proposition}\label{Prop.1.6}
Let $D\subset\C^n$ be a bounded domain with $C^1-$boundary, $A\subset\partial D$ compact. If $J_z=J^n_z$ for all $z\in D$ then 
$$\omega(z,A,D)=\omega^n(z,A,D).$$ 
\end{proposition}
\begin{proof}
 We know that  $\omega(.,A,D)\leq\omega^n(.,A,D)$. Let us  prove that $\omega^n(.,A,D)\leq\omega(.,A,D)$. Let $u$ be in the family defining $\omega^n$. 

 Set $g=-\chi_A.$ Note that $u^n\leq g$ on $\overline D.$ For $z\in D$ one has $$u^n(z)\leq\inf\left\{\int g\, d\mu, \mu\in J_z^n\right\}=\inf\left\{\int g \,d\mu, \mu\in J_z\right\},\,\mbox{ because } J_z=J_z^n.$$
Because $g$ is lower semicontinuous on $\overline D$, Edwards' theorem (Corol. 2.2 in \cite{Wi}) gives 
$$u^n(z)\leq\inf\left\{\int g\, d\mu, \mu\in J_z\right\}=\sup\left\{v(z), v\in\PSH(D)\cap\USC(\overline D), v\leq g\right\}\leq\omega(z,A,D).$$
As $u$ was taken arbitrarily in the family defining $\omega^n$ we infer that $\omega^n(z,A,D)\leq \omega(z,A,D)$ for all $z\in D.$ Thus $\omega(.,A,D)=\omega^n(.,A,D).$
\end{proof}
\begin{Remark} Notice that the dual of $J^n$ i.e $\{u^n, u\in\PSH(D), \sup_D u<+\infty\}$ is not a convex cone, indeed if $u$ and $v$ are bounded plurisubharmonic functions in $D$ we do not have in general $(u+v)^n=u^n+v^n$ in $\overline D$, cf.~\cite[Definition 2.2 and below]{Wi2}.
We use the class $J^n$ only to obtain an inequality. 
\end{Remark}
The  proof above applies to the next two propositions.
\begin{proposition}\label{Prop.3.9}
Let $D\subset\C^n$ be a bounded strongly star shaped domain and $A\subset\partial D$ compact. If $J_z=J^R_z$ for all $z\in D$ then 
$$\omega(z,A,D)=\omega^R(z,A,D).$$ 
\end{proposition}

\begin{proposition}\label{Prop.1.4}
Let $D\subset\C^n$ be a bounded domain and $A\subset\partial D$ compact. If $J_z=J_z^c$ for all $z\in D$ then $\omega(z,A,D)=\omega^c(z,A,D)$ for $z\in D$.
\end{proposition}
For $z\in \overline D$ define $$J^*_z=J^*_z(\overline D)=\{ \mu\in M(D) : u^*(z)\leq \int_{\overline D} u^*\,  d\mu \text{ for all } u\in\PSH(D),\, \sup_D u<+\infty \}. $$
 
Note that in \cite{Wi} the author worked with $J^*$ but this class does not represent a convex cone, see \cite{DiWi} for details.
Here we work with  $J$ instead of $J^*$  because of two reasons: firstly $\PSH(D)\cap \USC(\overline D)$  is a convex cone so Edwards' theorem can be applied, secondly for $z\in D$ and $g$ lower semicontinuous we have $$\inf\left\{\int g\, d\mu, \mu\in J_z^*\right\}=\inf\left\{\int g\, d\mu, \mu\in J_z\right\},$$ thus the results in \cite{Wi} remain valid for  $J_z$ if $z$ is an interior point.

\begin{corollary}
If $D$ is B-regular or if $D$ is strongly star shaped with respect to the origin or if $D$ is a polydisc then  $\omega(z,A,D)=\omega^c(z,A,D)$ for $z\in D.$
\end{corollary}
\begin{proof}
In these domains $J_z=J_z^c$  for $z\in D$ see \cite[Thm. 4.10, Thm. 4.11, Cor. 4.3 ]{Wi}. Then Proposition \ref{Prop.1.4} gives the result.
\end{proof}

For $H$ a positive definite hermitian $n\times n$-matrix, let $\rho_H(z)=\overline z^T Hz$ on $\C^n$ and set $D_H=\{z\in \C^n :\rho_H(z)<1 \}$.

\begin{proposition}\label{Prop.1.3}
On $D_H$ we have $J^n_z=J^c_z=J_z$   for all $z\in D_H$.
\end{proposition}

\begin{proof} Set $D=D_H.$ Let $z\in D$. Because for $u\in\PSH(D)\cap C(\overline D)$, $u=u^n$ on $\overline D$, we have $J^n_z\subset J^c_z$. Let $\mu\in J^c_z$  and $u\in\PSH(D)\cap\USC(\overline D)$.   Let $0<r<1$. Observe that in case of $D_H$ the map  $n(.,r)$ is holomorphic and maps $\overline D$ into $D$. Set $u_r(y)=u\circ n(y,r)$, $y\in\overline D$. Then $u_r$ is plurisubharmonic in a neighborhood of $\overline D$, hence  $u_r$ can be approximated monotonically from above by functions in $\PSH(D)\cap C(\overline D)$. By the monotone convergence theorem $u_r(z)\leq\int u_r\,d\mu$ for all $r\in ]0,1[$.
By Fatou's lemma  $$\limsup_{r \rightarrow 0} u_r(z)\leq\limsup_{r\rightarrow 0}\int_{\overline D}u_r(y)\,d\mu.$$
 For $y\in D$ one has $\limsup_{r\rightarrow 0} u_r(y)=u^n(y)$.  Because the set $[0,1]$ is not thin at $0$, see Theorem 2.7.2 in \cite{Ki}, we have $$u^n(z)=u(z)=\limsup_{r \rightarrow 0} u_r(z)\leq\int \limsup_{r\rightarrow 0} u_r(y)\,d\mu\leq  \int_{\overline D}  u^n(y)\,d\mu.$$
 Thus $\mu\in J^n_z$ it follows that $J^c_z\subset J^n_z$. Hence $J^c_z=J^n_z\subset J_z\subset J_z^c.$
 \end{proof}
 
The unit ball, i.e. the case where $H=\Id$, was done in \cite{Wi}. Our proof is a slight modification of Wikstr\"om's.
\begin{theorem}\label{corol.3.9} 
For all $z\in D_H$ one has $\omega(z,A,D_H)=\omega^n(z,A,D_H) =\omega^R(z,A,D_H)  =\omega^c(z,A,D_H) $ for all $A\subset\partial D_H$ compact.
\end{theorem}
\begin{proof} 
By Proposition \ref{Prop.1.3}  $J^c=J^n=J$  and  by Proposition \ref{Prop.1.6} and Proposition \ref{Prop.1.4}  $\omega^c=\omega^n=\omega$. As $D_H$ is strongly star shaped with respect to the origin, $J=J^R$ see Prop.\,5.4 in \cite{Wi} and  by  Proposition \ref{Prop.3.9} above, the equality  $\omega=\omega^R$ follows.
\end{proof}
\section{Circular sets}
Our goal in this section is to generalize Theorem 2.11 in \cite{IbJa} and solve Sadullaev's problem for circular sets in circular, strongly star shaped, (hence balanced) domains.
 \begin{theorem}
 Let $D$ be a bounded smooth circular domain that is strongly star shaped with respect to the origin  and  let $A\subset\partial D$ be a circular compact set. Then  
 \[\omega^n(.,A,D)\le\omega^R(.,A,D)=\omega^c(.,A,D).\]
 In particular, 
 \[ \omega_1(z,A,D)=\omega_2(z,A,D)=\omega_3(z,A,D).\]
 \end{theorem}
 \begin{proof} Let $u$ be in the family defining $\omega^n(.,A,D)$. Let $\rho$ be a smooth defining function for $D$ such that for all $\theta$ and $y$ in a neighborhood of $\overline D$ we have $\rho(y)=\rho(e^{i\theta} y)$.  For $0<t<1$ consider the function 
 \[ v_t(z,w)=u(n(w,t)z), \quad (w\in \overline D, \ z\in \C, |z|\le 1).\]
 For fixed $t$, $w$ the function $v_t(.,w)$ is subharmonic on the (closed) unit disc.
 Observe that $n(w,t) e^{i\theta}= n(e^{i\theta}w, t)$, so that for each $w\in A$ and all $\theta$
 \[\limsup_{t\downarrow 0} v_t(e^{i\theta}, w)\le -1.\]
 Hence for all $|z|\le 1$, $\limsup_{t\downarrow 0} v_t(z, w)\le -1$. It follows that $u(wz)\le -1$ for $w\in A$ and $|z|\le 1$. We infer that $u$ belongs to the family defining $\omega^R(.,A, D)$ and the inequality is proved.
 
 Now suppose that $u$ belongs to the family defining $\omega^R(.,A,D)$. Then $u(wz)\le-1$ for $w\in A$ and $|z|<1$. Therefore, for $0<r<1$ $u_r(w)=u(rw)$ is a plurisubharmonic function in a neighborhood of $\overline D$ that is less that $-1$ on $A$. Now $u_r$ can be approximated from above on $\overline D$ by a decreasing sequence $\{v_j\}$ of continuous PSH-functions. By Dini's theorem, for every $\epsilon>0$ there is a $j_0$ so that $v_j\le -1+\epsilon$ on $A$ hence also on a neighborhood of $A$. It follows that $u_r\le \omega^c(.A,D)$, and then also $u\le \omega^c(.A,D)$.
\end{proof}

 \section*{Appendix}
We attempted to apply the non-compact version of Edwards' duality theorem stated in \cite{GoPePo} to prove equalities for boundary extremal functions. However, we noticed that this version of Edwards' theorem as stated, does not hold. This appendix contains some counterexamples.
 
 Let $D \subset \C^n$ be a bounded set and $F\subset C(D)$ be a convex cone containing constants. $\B(D)$ denotes the set of Borel probability measures  with compact support in $D.$    For $z\in D$ set
$$J^F_z(D)=  \{ \mu \in \B(D), \mbox{ supp }\mu\subset D,  \;  u(z)\leq \int_{ D} u \, d\mu \mbox{ for all } u\in F\}.$$ In case $D$ is a bounded domain we make use of the notation in Section 2, and for $z\in\overline D$, we set $J_z^c=J_z^c(\overline D)$ and $J_z=J_z(\overline D)$.  Let $g: D\rightarrow \R$ and define $$Sg(z)=\sup\{u(z), u\in F, u\leq g\}$$ and  $$Ig(z)=\inf\{\int_{ D}g\, d\mu,\; \mu\in J_z^F(D)\}.$$
The following theorem is due to Edwards, see \cite[Theorem 2.1]{Ra}.

\begin{theorem}[Edwards]\label{Theo4.1} Assume that $D$ is compact and  $g $ is a bounded Borel function on $ D$, then $Sg(z)\leq Ig(z)$. If $g$ is lower semicontinuous , then $Sg=Ig$. 
\end{theorem}
 Edwards' theorem  is  very delicate. For instance if the kernel $g$ is merely upper semicontinuous, the theorem may fail, see \cite{Ra, GoPePo}. We will show that the theorem  may also fail if the set $D$ is not compact, contrary to the following theorem, which was formulated in (\cite[Thm.1.3]{GoPePo}).
\begin{theorem}[\cite{GoPePo}]\label{Theo4.2}
Let $D$ be a locally compact Hausdorff space countable at infinity. If $g\in C(D)$ then either $Sg\equiv-\infty$ or 
 $$Sg(z)=\inf\{ \int_{ D} g\,  d\mu , \mu\in J^F_z(D)\} .$$
\end{theorem}
  However, this result does not hold if $D$ is open.

\begin{counterexample}
For the sake of finding a contradiction, assume that Theorem \ref{Theo4.2} holds for  all open set $D'$ i.e
\begin{align} \label{eqcount}
\sup\{u(z), u\in F, u\leq g\}=\inf\{ \int_{ D'} g\,  d\mu , \mu\in J^F_z(D')\}, 
\end{align}
where $z\in D'$,  $g\in C(D')$,  $F\subset C(D')$ is a convex cone containing the constants. 

Let  $D$  be a bounded open ball and $V\subset\subset D$  be an open  ball. Define $$u_{D,V}(z)=\sup\{u(z), u\in\PSH(D), u\leq-\chi_V\}.$$ Let  $u\in\PSH(D)^-$ so that the set  $\{u=-\infty\}$ is dense in $V.$  For $m>0$ set $U_m=\{\frac{u}{m}<-1\}\cap V$,  and $F=\PSH(D)\cap C(\overline D)$. Observe that the function $g_m=-\chi_{U_m}$ is continuous in the open set $D\setminus\partial U_m$ and that $F$ is a convex cone in $C(D\setminus \partial U_m)$ containing the constants. By \eqref{eqcount} we obtain for  $z\in D\setminus\partial U_m $ the following equality (we take for $D'$ the set $D\setminus\partial U_m$)
$$\inf\left\{\int_{D\setminus\partial U_m} g_m\, d\mu, \mu\in J_z^F(D\setminus\partial U_m)\right\}=\sup\{v, v\in F, v\leq g_m\} \mbox{ on } D\setminus\partial U_m.$$
If  $v\in F$ and $v\leq g_m$, then $v\leq-\chi_{V}$ because $\overline U_m=\overline V$ implies $v\leq u_{D,V}$, hence
$$\inf\left\{\int_{D\setminus\partial U_m} g_m\, d\mu, \mu\in J_z^F(D\setminus\partial U_m)\right\}=\sup\{v, v\in F, v\leq g_m \mbox{ on } D\setminus\partial U_m\}\leq u_{D,V}.$$
 As $J^F_z(D\setminus\partial U_m)\subset J_z^c$ we have  on $D\setminus\partial U_m$ 
$$\inf\left\{\int_{\overline D} g_m\, d\mu, \mu\in J_z^c\right\}\leq   \inf\left\{\int_{D\setminus\partial U_m} g_m\, d\mu, \mu\in J_z^F(D\setminus\partial U_m)\right\}\leq u_{D,V}.$$
Because $D$ is a ball, by \cite[Cor.4.3]{Wi} $J_z=J_z^c$. It follows that 
$$\inf\left\{\int_{\overline D} g_m\, d\mu, \mu\in J_z\right\}=  \inf\left\{\int_{D} g_m\, d\mu, \mu\in J_z^c\right\}\leq u_{D,V} \mbox{ on }D\setminus\partial U_m.$$
Now $\frac{u}{m}$ is plurisubharmonic and $\frac{u}{m}\leq g_m$, hence for all $m>0$ one has 
$$ \frac{1}{m}u(z)\leq   \inf\left\{\int_{\overline D} g_m\, d\mu, \mu\in J_z\right\} \leq u_{D,V}(z)\mbox{ for } z\in  D\setminus\partial U_m.$$
As $D\setminus\overline V\subset D\setminus\partial U_m$ we have for all $m>0$ that $$\frac{u}{m}\leq u_{D,V} \mbox{ on } D\setminus\overline V.$$
This is impossible since       $$0\equiv\left(\sup_m\frac{u}{m}\right)^*\leq u_{D,V} <0\text{ on } D\setminus\overline V.$$
The conclusion is that equality \eqref{eqcount}  is false in  open sets $D'$.
\end{counterexample}

Next we prove that the version of Edwards' theorem  stated in Theorem \ref{Theo4.1} does not hold for (open) B-regular domains, i.e. connected open sets.

\begin{counterexample}
Let $D$ be a bounded B-regular domain and $V\subset\partial D$ be relatively open. Then $\overline V$ is not b-pluripolar, see Propositions 3.5 and 2.4 in \cite{IbJa}.
There exists a countable $L\subset D$ so that  $\overline L=L\cup \overline V$
is  compact  in $\overline D$ 
cf., \cite[Lemma 4.3]{IbJa}. Set $g=-\chi_{L}$ and $F=\PSH(D)\cap C(\overline D)$.  As  $L$ is non-empty and does not have any accumulation point in $D$, $g$ is lower semicontinuous in $D$ . If Theorem \ref{Theo4.1} would hold for $D$ we would get for $z\in D$   
\begin{align*}
&\inf\left\{ \int_{ D} g \, d\mu , \mu\in J^F_z(D)\right\}=\sup\{u(z), u\in F, u\leq g\}\leq \omega(z, V,D),\\
&\inf\left\{ \int_{ D} g\,  d\mu , \mu\in J^c_z\right\} \leq \inf\left\{ \int_{ D} g\,d\mu , \mu\in J^F_z(D)\right\} \leq \omega(z,V,D)   \intertext{  because  $J_z^F(D)\subset J_z^c$,}\
&\inf\left\{ \int_{ D} g\,  d\mu , \mu\in J_z\right\} = \inf\left\{ \int_{ D} g\,  d\mu , \mu\in J^c_z\right\} \leq \omega(z,V,D) \intertext{ because $ J_z=J_z^c$,}\
& \sup\{u(z), u\in\PSH(D), u\leq g\} \leq \inf\left\{ \int_{ D} g\,  d\mu , \mu\in J_z\right\} \leq \omega(z,V,D).
\end{align*}
 Finally, because $L$ is countable and therefore pluripolar, we would get 
$$ 0=(\sup\{u(z), u\in\PSH(D), u\leq g\})^* \leq \omega(z,V,D).$$
This is impossible since $V$ is not b-pluripolar. The conclusion is that Edwards' theorem does not hold in $D$.
\end{counterexample}

\begin{Remark} Approximating $g$ by continuous functions one can show that Theorem \ref{Theo4.2} does not hold in B-regular domains. 
\end{Remark}
These counterexamples make it unlikely that a useful non-compact version of Edwards' theorem can be found. We have not been  able to pinpoint the problematic points in (\cite[Thm.1.3]{GoPePo}).

\smallskip\noindent
\emph{Acknowledgement.} 
The first author is  grateful to Prof. A. Edigarian and Prof. L. Kosi\'nski for  supports. We are grateful to Peter Pflug who spotted a mistake in an earlier version of this paper. We thank the referee for his comments which have  made the paper clearer. 


\begin{thebibliography}{99}

\bibitem {DiWi}  N. Q. Dieu and F. Wikstr\"om, \textit{Jensen Measure and Approximation of Plurisubharmonic Functions},   Michigan  Math. J. \textbf{53} (2005) 529--544.


\bibitem{IbJa} I. K. Djire and J. Wiegerinck, \textit{ Characterizations of boundary pluripolar hulls}, Comp. Var. and Ellip. Equat. \textbf{61} (2016) no. 8, 1133--1144.

\bibitem {EdSi}A. Edigarian and R. Sigurdsson,  \textit{  Relative extremal functions and characterization of pluripolar sets in complex manifolds}, Trans. Amer. Math. Soc. \textbf{362} (2010) 5321--5331.

\bibitem {EdSig}A. Edigarian and R. Sigurdsson,  \textit{The Relative Extremal Function for Borel sets in Complex Manifolds}, (2006) arXiv:math/0607313v1 [math.CV].

\bibitem {Edi}A. Edigarian, \textit{Polynomial hulls and proper analytic disks}. Ark. Mat. \textbf{50} (2012), no. 1, 59--67. 

\bibitem {GoPePo} N. G. Gogus, T. L. Perkins and E. A. Poletsky, \textit{Non compact versions of Edwards theorem}, Positivity \textbf{17} (2013) 459--473.

\bibitem  {Ki} M. Klimek,  \textit{Pluripotential Theory}, London Math. Soc., Monographs New Series 6, Clarendon Press, Oxford, 1991.
\bibitem {Po} E. A. Poletsky,  \textit{Holomorphic currents}, Indiana Univ. Math. J. \textbf{42} (1993)  85--144.

\bibitem {Ra}   T. J. Ransford, \textit{Jensen measures}, In: Approximation, Complex Analysis, and Potential Theory,  (Montreal, QC, 2000), 221–237, NATO Sci. Ser. II Math. Phys. Chem., \textbf{37}  Kluwer Acad. Publ., Dordrecht, 2001.


\bibitem {Sa} A. Sadullaev, \textit{Plurisubharmonic measures and capacities on complex manifolds}, (Russian) Uspekhi Mat. Nauk \textbf{36} (1981) no. 4, 53--105. English translation in \textit{Russian Math. Surveys}, \textbf{36} (1981), no. 4, 61--119.

\bibitem {Wi}F. Wikstr\"om,  \textit{Jensen measure and boundary values of plurisubharmonic functions}, Ark. Mat. \textbf{39} (2001) 181--200.

\bibitem{Wi2}
F. Wikstr\"om, \textit{ The Dirichlet problem for maximal plurisubharmonic functions on analytic varieties in $\C^n$},  Internat. J. Math. \textbf{20} (2009) no. 4, 521–-528.


 \end{thebibliography}
 \end{document}